\documentclass[12pt]{amsart} 

\usepackage{amsmath,amsfonts,amssymb,mathtools,amsthm}
\usepackage{mathrsfs}

\theoremstyle{plain}

\newtheorem{theorem}{Theorem}[section]
\newtheorem{lemma}[theorem]{Lemma}

\newtheorem{corollary}[theorem]{Corollary}

\newtheorem{proposition}[theorem]{Proposition}

\theoremstyle{definition}

\newtheorem{hypothesis}[theorem]{Hypothesis}

\numberwithin{equation}{subsection}

\newcommand{\cc}{{\mathbb C}}
\newcommand{\pp}{{\mathbb P}}
\newcommand{\rr}{{\mathbb R}}
\newcommand{\qq}{{\mathbb Q}}
\newcommand{\zz}{{\mathbb Z}}
\newcommand{\nn}{{\mathbb N}}
\newcommand{\ff}{{\mathbb F}}

\newcommand{\Gm}{{\mathbb G}_{\rm m}}
\newcommand{\aaaa}{{\mathbb A}}

\newcommand{\hhh}{{\mathbb H}}

\newcommand{\Oo}{{\mathcal O}}

\newcommand{\Rr}{{\mathcal R}}

\newcommand{\Vv}{{\mathcal V}}

\newcommand{\Ll}{{\mathcal L}}
\newcommand{\Ee}{{\mathcal E}}

\newcommand{\srD}{{\mathscr D}}

\newcommand{\srR}{{\mathscr R}}

\newcommand{\srG}{{\mathscr G}}
\newcommand{\srM}{{\mathscr M}}
\newcommand{\srU}{{\mathscr U}}

\newcommand{\srMtilde}{\widetilde{\srM}}
\newcommand{\Mtilde}{\widetilde{M}}

\newcommand{\eop}{\hfill$\Box$}

\title[Twistor geometry of parabolic structures]{The twistor geometry of parabolic structures in rank two}
\author{Carlos Simpson}
\date{} 



\begin{document}

\begin{abstract}
Let $X$ be a quasi-projective curve, compactified to $(Y,D)$ with $X=Y-D$. 
We construct a Deligne-Hitchin twistor space out of moduli spaces of 
framed $\lambda$-connections of rank $2$ over $Y$ with logarithmic singularities and quasi-parabolic structure along $D$. 
To do this, one should divide by a Hecke-gauge groupoid. 
Tame harmonic bundles on $X$ give preferred sections, and the relative tangent bundle along a preferred section has a mixed 
twistor structure with weights $0,1,2$. The weight $2$ piece corresponds to the deformations of the KMS structure
including parabolic weights and the residues of the $\lambda$-connection. 
\end{abstract}

\maketitle



\section{Introduction}

In \cite{Hitchin}, Hitchin gave a hyperkähler structure on the moduli space of local
systems over a smooth compact Riemann surface. By Penrose theory, this leads to 
a {\em twistor space}. Deligne gave an interpretation of the construction of the twistor
space in terms of the moduli space of $\lambda$-connections. 
This viewpoint 
is amenable to generalization to the case of quasiprojective varieties. 
For rank $1$ local
systems on an open curve, a weight $2$ property  
for the local monodromy transformations around the punctures came into view \cite{weighttwo}, and this
was related to parabolic structures. 

In the present paper, we would like to consider how to move to higher rank local systems on an open curve. 
We'll look at the case of rank $2$ bundles with logarithmic connection. The fundamental picture
is an interplay between the notions of quasi-parabolic structure and parabolic structure, that
were introduced by Seshadri in \cite{Seshadri} and developed further in \cite{MehtaSeshadri} and the
subsequent literature.

Before stating the results in \S \ref{secresults}, 
we'll first have a look at the classical twistor space theory, the original case of a compact Riemann surface,
then the rank $1$ case over a quasiprojective curve. 

\subsection{Twistor space}

Let $\hhh = \rr \langle 1,I,J,K\rangle$ be the algebra of quaternions. The
space of complex structures $\kappa \in \hhh , \;\; \kappa ^2 = -1$ is identified with the two-sphere $S^2$
and provided with a complex structure making it into $\pp^1$. 
With coordinate on the main chart $\aaaa^1$
denoted by $\lambda$, the complex structure $I$ corresponds to $\lambda = 0$ and $J$ corresponds to  $\lambda = 1$. 
The antipodal involution $\kappa \mapsto -\kappa$ corresponds to $\sigma : \lambda \mapsto -\overline{\lambda}^{-1}$,
it is a real structure on $\pp^1$ with empty set of real points. 

Suppose $V$ is an $\hhh$-module i.e.\  a quaternionic vector space.
The product $V\times \pp^1$ has a global complex structure inducing $\kappa$ on $V\times \{\kappa \}$, 
making it into the total space of a vector bundle $\Vv/\pp^1$. 
We call the sections of $\Vv$ of the form $\{ v\} \times \pp^1$ the {\em preferred sections} . 

\begin{proposition}
\label{linearprop}
This {\em twistor space} construction is an equivalence between finite dimensional quaternionic vector spaces, and
bundles $\Vv / \pp^1$, provided with an involution covering $\sigma$, such that $\Vv$ is semistable
of slope $1$ (i.e.\ $\Vv \cong \Oo _{\pp^1}(1)^{\oplus 2d}$). The underlying $\rr$-vector space
is recovered as $V = \Gamma (\pp^1,\Vv ) ^{\sigma}$.
For each $\kappa \in \pp^1$, the projection $V\rightarrow \Vv _{\kappa}$ is an isomorphism of real vector spaces 
inducing the complex structure $\kappa$ on $V$. 
\end{proposition}

This linear picture extends to the nonlinear case \cite{HKLR}. If $M$ is an integrable quaternionic manifold $M$, its {\em twistor space} is 
$$
Tw(M):= M\times \pp^1
$$
with complex structure and antilinear involution $\sigma $ obtained in the same way. 
The horizontal ``preferred sections'' $\{ m\} \times \pp^1$ are holomorphic, $\sigma$-invariant, 
and for any $\kappa \in \pp^1$ the fiber $M\times \{ \kappa \}$ is given
the complex structure $M_{\kappa}$ determined by the action of $\kappa \in \hhh$ on the tangent spaces of $M$. 
The ``integrability'' condition says that these are integrable complex structures, and the general Penrose theory yields an
integrable complex structure on the total space $Tw(M)$. 

The relative tangent bundle along a preferred section is the vector bundle given by Proposition \ref{linearprop}, in particular it
is semistable of slope $1$---we say it has {\em weight $1$}. Locally around a preferred section, the $\sigma$-invariant sections are all
preferred sections, and a neighborhood in the space of these sections maps isomorphically
to a neighborhood in any one fiber $M_{\kappa}$ thanks to the weight $1$ property. These however
can fail for general $\sigma$-invariant sections not assumed to be near to preferred ones. A general study is given in  \cite{BeckBiswasHellerRoser,BeckHellerRoser,BiswasHellerRoser}.

\subsection{Hitchin's twistor space in the compact case} 

The moduli space $M$ of local systems on a smooth
compact Riemann surface $X$ has a quaternionic and indeed hyperkähler structure \cite{Hitchin}, the latter meaning that there is a Riemannian
metric Kähler for all the complex structures. We obtain the twistor space $Tw(M)$. In what follows, we look only at the smooth
points of the moduli space without further mention.

For the complex structure $\kappa = I$ corresponding to $\lambda = 0$ in $\pp^1$,
the complex moduli space $M_0$ is the moduli space of Hitchin pairs or ``Higgs bundles'' $(E,\varphi )$ on $X$. 
We may call this the {\em Dolbeault moduli space} denoted $M_{\rm Dol}$ in view of the analogy with Dolbeault cohomology. 
For the complex structure $\kappa =J$ corresponding to $\lambda = 1$ in $\pp^1$,
the complex moduli space $M_1$ is the moduli space of vector bundles with integrable connection $(E,\nabla )$ on $X$. 
We call this the {\em de Rham moduli space} denoted $M_{\rm dR}$ in view of the analogy with de Rham cohomology.

Furthermore, for all the complex structures $\kappa$ corresponding to $\lambda \neq 0,\infty$ the moduli spaces
$M_{\kappa}$ are naturally isomorphic, so they are all isomorphic to $M_{\rm dR} = M_1$, 
which is in turn analytically isomorphic to the ``Betti'' moduli space of
local systems or representations of the fundamental group.

Deligne, having discussed with Witten, gave a reinterpretation of $Tw(M)$ as follows. Each $M_{\lambda}$ is the moduli space of 
{\em vector bundles with $\lambda$-connection} $(E,\nabla )$. For $\lambda \neq 0$ the rescaling $\lambda ^{-1}\nabla$
is just a connection, yielding the isomorphisms refered to above, whereas for $\lambda = 0$ a $\lambda$-connection is
the same thing as a Higgs field $\varphi$.

We may make an algebraic geometry construction of the family of moduli spaces over $\aaaa^1\subset \pp^1$, which
for reasons of analogy with the Dolbeault and de Rham terminology, we call $M_{\rm Hod}$ for Hodge. 
This space together with its $\cc^{\ast}$-action may be 
viewed as the ``Hodge filtration'' relating de Rham to Dolbeault \cite{hfnac}.

The isomorphisms between different nonzero $\lambda \in \aaaa^1-\{ 0\} = \Gm$ fit together to give an
analytic trivialization
$$
\left( \Gm  \times _{\aaaa^1} M_{\rm Hod}  \right) ^{\rm an}\cong 
\left( \Gm \times M_B \right) ^{\rm an}
$$
where $M_B$ is the ``Betti'' moduli space of representations of the fundamental group.

Then, the condition of existence of an antipodal involution covering $\sigma$ motivated Deligne to define a {\em glueing} between
$M_{\rm Hod}(X)$ and  $M_{\rm Hod}(\overline{X})$ using the isomorphism
$$
\pi _1(X)\cong \pi _1(\overline{X}) \;\;\; \mbox{ whence } \;\;\; M_B(X)\cong M_B(\overline{X})
$$
and applying the involution $\lambda \mapsto -\lambda ^{-1}$ on $\Gm$. 
Glueing the two pieces together yields a space 
$$
M_{\rm Hod}(X) \cup M_{\rm Hod}(\overline{X}) \;\; =: \;\; 
M_{\rm DH} \rightarrow \pp^1
$$
and one can define an antipodal involution using that $\overline{X}$ is the complex conjugate variety to $X$ and the fact that the
moduli space $M_{\rm Hod}$ is a canonical algebraic geometry construction so it supports a complex conjugation operation.

\begin{proposition}
The Deligne-Hitchin moduli space constructed by Deligne's glueing is isomorphic to the twistor space for Hitchin's quaternionic 
structure:
$$
\begin{array}{ccccc}
M_{\rm DH} & &\cong && Tw(M)\\
& \searrow &&\swarrow & \\
&& \pp^1 && 
\end{array}
$$
\end{proposition}

The preferred sections of the twistor space correspond to sections of the fibration $M_{\rm DH} \rightarrow \pp^1$
that we also call ``preferred sections''. These are maps $\pp^1\rightarrow M_{\rm DH}$ that are obtained whenever 
we have a {\em harmonic bundle}
$$
(E,\partial, \overline{\partial}, \varphi , \varphi ^{\dagger},h )
$$
corresponding to a solution of Hitchin's equations. 
For $\lambda \in \aaaa^1$ the point in the moduli space
of holomorphic vector bundles with $\lambda$-connections is 
$$
\left( 
\Ee _{\lambda} := (E,\overline{\partial} + \lambda \varphi ^{\dagger}), \;\;
\nabla _{\lambda} := \lambda \partial + \varphi \right) .
$$
These preferred sections are compatible with the antipodal involution. 
The fact that this construction gives the twistor space of a quaternionic manifold comes from the following 
property:

\begin{proposition}
Suppose $\rho : \pp^1\rightarrow M_{\rm DH}$ is a preferred section defined as coming from a harmonic bundle in the
above way. Let 
$$
T_{\rho} =\rho ^{\ast} T(M_{\rm DH} / \pp^1 )
$$
be the pullback of the relative tangent bundle, or equivalently  the normal bundle of $M_{\rm DH}$ to the section.
Then $T_{\rho}$ is a semistable vector bundle of slope $1$ over $\pp^1$. 
\end{proposition}
\eop

An analogy with Hodge structures motivates us to call the property of being semistable of slope $1$, a property of {\em weight $1$}.
Thus, the fact that we have a quaternionic structure on the moduli space comes from a weight $1$ property
for the tangent bundle to the preferred section. 

Fundamentally, the calculation going into the proof uses the observation that the tangent space of the moduli space is an $H^1$,
calculated by some kind of harmonic forms. Then, the fact that they are $1$-forms means that the transition functions
needed to pass from the $\aaaa^1$ neighborhood of $\lambda = 0$ to the 
$\aaaa^1$ neighborhood of $\lambda = \infty$ involve $\lambda ^{-1}$ leading to the semistable of slope $1$ property.

This weight $1$ property is the nonabelian cohomology analogue of the statement in usual Hodge theory that
$H^1_{\rm dR}(X)$ has a weight $1$ Hodge structure, and similarly for a variety over $\ff _q$
that the étale cohomology $H^1_{\rm et}(X_{\overline{\ff _q}}, \qq _{\ell})$
has weight $1$ in the sense that the eigenvalues of Frobenius have norm $q^{1/2}$.

\subsection{A weight two property in the quasiprojective case}

In usual cohomology, recall that for 
$X=\pp ^1- \{ 0,\infty \}$
the mixed Hodge structure on $H^1(X)$ is one-dimensional, of weight $2$, and in 
the arithmetic setting 
$H^1_{\rm et}(X, \qq _{\ell} ) \cong \qq _{\ell} (1)$
is a Tate twist having weight two. 
The weight two property is localized at the punctures, for example 
from arithmetic geometry the inertia group has the form of a Tate twist, so it has weight $-2$ and the space
of representations of the inertia group should be thought of as having weight $2$. 

Therefore, one may naturally conjecture that the weight $1$ property for the
twistor space would become a weight $2$ property for the local monodromy transformations around punctures. 
Pridham \cite{Pridham} indeed does this for deformations of $\ell$-adic representations. 
For the case of moduli spaces of local systems of rank $1$ over an
open curve, this was discussed in the paper \cite{weighttwo}.

Let's review some of the details under the simplifying assumption that 
$X=\pp^1 - \{ 0,\infty \}$. The only data of a local system is then a single local monodromy at 
a puncture. A line bundle on $X$ is trivial, and a logarithmic $\lambda$-connection has the form
$$
\nabla (\lambda , a) = \lambda d + a\frac{dz}{z} .
$$
There is an action of change of the trivialization making $(\lambda , a)$ equivalent to $(\lambda , a + k\lambda )$ for
any $k\in \zz$. The singularity of this action at $\lambda = 0$ is one of the difficulties of the open curve situation.

Let $\srG \cong \zz$ be this ``gauge group'' acting. We note that the gauge transformations may also be viewed as
Hecke transformations at the singular points. In this example, to maintain a trivial bundle we do simultaneous
Hecke transformations at both points, one up and one down. 

The moduli space may be described as 
$$
M_{\rm Hod}  := \aaaa^1 \times \cc / \srG 
$$
using $\aaaa^1$ for the $\lambda$ variable and $\cc$ for the coefficient $a$. 
Note that the group acts discretely over $\lambda \neq 0$ 
but the  stabilizer group of the fiber over $\lambda = 0$ is the full $\srG = \zz$.
It is therefore not completely clear what kind of structure best to accord to the quotient.
More on this aspect later.

The Riemann-Hilbert correspondance over $\lambda \neq 0$ 
sends 
the connection $\lambda ^{-1} \nabla (\lambda , a)$ to the monodromy around the loop generating $\pi _1(X)$:
$$
\Gm \times \cc / \srG \stackrel{\cong}{\longrightarrow} \Gm \times \cc^{\ast}
\;\;\;\;
(\lambda , a)\, \mapsto \,(\lambda , {\rm exp}(2\pi i a / \lambda ) ).
$$
We would like to use this identification
to glue $M_{\rm Hod}$ to the other piece in the Deligne glueing. Let $\mu := -\lambda ^{-1}$ denote the coordinate of the
other chart $\aaaa^1\subset \pp^1$. A point 
$(\mu , b)\in M_{\rm Hod}(\overline{X})$
has monodromy transformation ${\rm exp}(2\pi i b / \mu )$ along the generating loop for $\pi _1(\overline{X})$.

The
topological isomorphism $X^{\rm top} \cong \overline{X}^{\rm top}$ takes the generator to minus the generator,
so the Deligne glueing should associate $(\lambda , a)$ with $(\mu , b)$ (up to the $\srG$ action) when
$$ 
{\rm exp}(2\pi i a / \lambda )  =
{\rm exp}(- 2\pi i b / \mu ) .
$$
Lifting over the action of the gauge group, this condition becomes
$$
a/\lambda = -b/\mu = -b / (-\lambda ^{-1} ) =\lambda b,\;\;\;\; \mbox{ i.e. } \;\; a = \lambda ^2 b.
$$
This is the glueing condition for the line bundle $\Oo _{\pp^1}(2)$ over $\pp^1$, 
yielding the weight $2$ expression for the Deligne-Hitchin space:
$$
M_{\rm DH} = {\rm Tot}(\Oo _{\pp ^1} (2) ) / \srG . 
$$
There is also a natural antipodal involution, and the {\em preferred sections} are the sections that are compatible with $\sigma$. 
Recall from the compact case that we wanted to look at the space of $\sigma$-equivariant sections of 
$M_{\rm DH} /\pp^1$. Here let's lift over the action of the gauge group and look at the space of $\sigma$-equivariant sections of 
$\Oo _{\pp^1}(2)$. Before asking for $\sigma$-equivariance we have 
$\Gamma (\pp^1,\Oo _{\pp ^1}(2)) \cong \cc^3$.

\begin{lemma}
$$
\Gamma (\pp^1,\Oo _{\pp ^1}(2))^{\sigma} \cong \rr^3.
$$
For any $\kappa \in \pp^1$ the restriction morphism from these sections to the fiber
$\cc _{\kappa} = \Oo _{\pp^1}(2) _{\kappa}$ is a surjection 
$$
\rr^3 \rightarrow \cc _{\kappa}.
$$
There is a natural splitting as $\rr ^3 \cong \rr \times \cc_{\kappa}$ such that the
generator of the gauge group has the form $(1,-\lambda (\kappa ))$. 
\end{lemma}
\eop

The extra real parameter, kernel of the restriction map,
turns out to be the {\em parabolic weight parameter}. 
If $(E,\partial, \overline{\partial}, \varphi , \varphi ^{\dagger},h )$ is a tame harmonic bundle 
of rank $1$ over $X$
then it yields a $\sigma$-invariant section, and for $\lambda \in \aaaa^1$ the corresponding
point in $\rr \times \cc$ is $(p,e)$ where $p$ is the parabolic weight and $e$ the eigenvalue of the
residue of the $\lambda$-connection.
The parabolic weight expresses the growth rate of the harmonic metric $h$ near a puncture. 
The fact that we have this extra real parameter may be seen as a manifestation of the
fact that the monodromy transformations around punctures lie in a space whose Hodge weight is $2$.

Recall that Mochizuki defines the notion of {\em KMS-spectrum} of a harmonic bundle on $X$ at a puncture $y\in D$. 
This is the set of residual data consisting of a parabolic weight and an eigenvalue of the residue, for the asymptotic
structure of the harmonic bundle near $y$. Each element of the KMS spectrum is a vector in the space $\rr^3$ that
we have seen above; such a point is interpreted as a pair $({\mathfrak p}, {\mathfrak e})$ consisting of a parabolic weight and an eigenvalue,
in a way that depends on $\lambda$.

Sabbah \cite{Sabbah} and Mochizuki \cite{Mochizuki} gave formulas for the
variation of parabolic weight ${\mathfrak p}$ and eigenvalue ${\mathfrak e}$ as a function
of $\lambda$, generalizing my formulas \cite{hbnc} for $\lambda = 1$.

We'll use the notations of \cite{Mochizuki} (Sabbah's notations are slightly different but equivalent).
Starting with $(a,\alpha ) \in \rr \times \cc $ at $\lambda = 0$, the parabolic weight of the
parabolic structure at $\lambda$, giving the growth rate of holomorphic sections for the
holomorphic structure $\partial + \lambda \varphi ^{\dagger}$, is
\begin{equation}
\label{frakp}
{\mathfrak p}(\lambda, (a,\alpha )) = a + 2{\rm Re}(\lambda \overline{\alpha}).
\end{equation}
The eigenvalue of the residue of the logarithmic $\lambda$-connection $\nabla _{\lambda} = 
\lambda \partial + \varphi$ is
\begin{equation}
\label{frake}
{\mathfrak e}(\lambda, (a,\alpha )) = \alpha - a\lambda  -\overline{\alpha} \lambda ^2.
\end{equation}
These formulas may be derived \cite{weighttwo} as a consequence of the weight $2$ twistor space interpretation,
with
the expression of $\rr \times \cc$, depending on $\lambda$, as corresponding to a unique $\rr^3$ independent of 
$\lambda$.

\subsection{The rank $2$ case}
\label{secresults}

We would like to extend this picture to higher rank local systems on an open curve. 
Let's look at some potential difficulties in light of the previous discussion. 
Many of these issues have been raised by Nitsure and other authors \cite{HaiEtAl,MisraSingh,Mochizuki,NitsureRS,NitsureSabbah}.

Although $\srG$ was a group acting in the above example, it is necessary in general to consider
the action of a groupoid, that we'll call the {\em Hecke-gauge groupoid}. 

A first observation is that the
action of this groupoid becomes singular over $\lambda = 0$, in the above example the entire $\srG = \zz$ stabilizes
the full fiber over $\lambda = 0$. 
For this reason, we'll tend not really to look at the quotient space, but to retain just the action groupoid instead. 
One possible solution here would be to invoke the notion of {\em diffeological space}.

A next observation is that we have sidestepped any discussion of stability. From the compact case recall that the construction
of $M_{\rm Dol}$ requires the notion of stability of a Higgs bundle, so this information is needed in the fiber over $\lambda = 0$
for the construction of $M_{\rm Hod}$. 
However, in the quasiprojective case, defining stability (needed \cite{InabaIwasakiSaito} for all values of $\lambda$) 
requires knowing the parabolic weights, but we are trying to recover the parabolic
weights from the twistor space construction itself as happened in rank $1$. 
Without a notion of stability, we are going to be getting moduli spaces that are not of finite type but only locally of finite type, 
even over $\lambda \neq 0$ \cite{Herrero}. 
This should be accepted.

A third difficulty for making the construction is that, from the formulas
\eqref{frakp}, \eqref{frake} for the variation of eigenvalue of the residue as a function of
$\lambda$, a preferred section is always going to have some points where the eigenvalues are resonant, so we are not
able just to impose a non-resonance condition on the residues. 
We do however impose some conditions on the fiber over $0$ so as to improve somewhat the moduli problem. This makes it so that
the discussion will work for most but not all ``preferred sections''. 

To get a mixed twistor property for the relative tangent space along a preferred section, we apply the theory of Sabbah \cite{Sabbah}
and Mochizuki \cite{Mochizuki}. Although we use only a small and early fraction of their theory, our application highlights some of the 
subtleties involved and might therefore serve as a gentle introduction. 

\medskip

Here now is a summary of what we do, stated in Theorems \ref{mainmod}, \ref{mainRH} and \ref{maintwistor}. Those will then be
proven in the subsequent sections.

Let 
$Y$ be a compact Riemann surface and $D\subset Y$ a reduced divisor. Set $X:= Y-D$, and choose a base-point $x\in X$. 
We look at bundles of rank $2$.

A framed quasi-parabolic bundle with logarithmic $\lambda$-connection is 
$$
(\lambda , E, \nabla , F,\beta )
$$
where $\lambda \in \aaaa^1$, $E$ is a rank $2$ vector bundle on $Y$, $\nabla$ is a logarithmic $\lambda$-connection i.e.
$$
\nabla : E \rightarrow E\otimes \Omega ^1_Y(\log D), \;\;\; \nabla (ae) = a\nabla (e) + \lambda d(a)e ,
$$
$F=F_y \subset E_y$ is a one-dimensional subspace preserved by ${\rm res}_y(\nabla )$, and 
$\beta : E_x\cong \cc^2$ is a framing over the base-point.

\begin{hypothesis}
\label{hypstar}
Let $(\lambda , E, \nabla , F,\beta )$ be a framed quasi-parabolic logarithmic $\lambda$-connection. 
We make the following hypotheses. 
\begin{enumerate}
\item 
The only quasi-parabolic endomorphisms of $( E, \nabla , F)$ are scalars. 

\item 
If $\lambda = 0$, the spectral curve of $\varphi = \nabla$ (which is a Higgs field in this case) is irreducible of degree $2$
over the base.   

\item 
When speaking of a harmonic bundle, we assume furthermore that the two KMS spectrum elements 
(parabolic weight, residue eigenvalue) modulo $\zz$ are distinct at each $y\in D$. This may be measured at $\lambda = 0$
in $(\rr / \zz ) \times \cc$. . 

\end{enumerate}
\end{hypothesis}

Note that (2)$\Rightarrow$(1) at $\lambda = 0$ since Higgs bundle corresponds to a rank $1$ torsion-free sheaf over the
spectral curve and its only endomorphisms are scalars. We also note that (1) implies directly that the framed object is
rigid, i.e.\  there are no nontrivial endomorphisms respecting the framing. We'll see in Lemma \ref{obslem} that (1) implies 
the moduli problem is unobstructed hence smooth.

\begin{theorem}
\label{mainmod}
There exists a separated algebraic space, smooth and locally of finite type, parametrizing the framed quasi-parabolic logarithmic 
$\lambda$-connections $(\lambda , E, \nabla , F,\beta )$
satisfying Hypothesis \ref{hypstar}
$$
\widetilde{M}_{\rm Hod}(X) \stackrel{\lambda}{\longrightarrow} \aaaa ^1.
$$
\end{theorem}

We would then like to define an equivalence relation identifying different logarithmic connections that correspond to the
same Betti local system data.

Let $\widetilde{M}_B(X)$ denote the 
moduli space 
of tuples $(L,F,\beta )$ where $L$ is a rank $2$ local system on $X$, $F=\{ F_y\}$ is 
a sub-local system of the restriction of $L$ to a punctured disk around $y$, and $\beta : L_x\cong \cc^2$ is
a framing. 
Impose the analogue of Hypothesis \ref{hypstar} (1), namely that the only endomorphisms
preserving the subspaces are scalars. 
Let $\widetilde{M}_B(X)$ be the moduli space for such framed quasi-parabolic local systems. 

Forgetting the framing at $x$ would provide a map to an open subset defined by \ref{hypstar} (1) of the
${\mathscr X}$ moduli space of Fock-Goncharov \cite{FockGoncharov}.

Define the {\em Betti gauge groupoid} $\srG _B$ acting on $\widetilde{M}_B(X)$ to be the groupoid
consisting of 
partially defined morphisms from $\widetilde{M}_B(X)$ to itself, generated by the operation:

\noindent
$(P)$: defined on the open set where the eigenvalues of the local monodromy transformation 
are distinct, sending $(L,F,\beta )$ to $(L,F^{\perp},\beta )$ where $F^{\perp}_y$ is the eigenspace 
of the local monodromy that is different from $F_y$. 

This defines an etale groupoid, with quotient 
$M_B(X)=\widetilde{M}_B(X) / \srG _B$ a non-separated algebraic stack. 
This quotient stack behaves as in Kollár's observation \cite{Kollar}, since the operation is partially defined and étale.

\begin{theorem}
\label{mainRH}
There is an étale groupoid 
$$
\srG_{\rm dR} \rightarrow \widetilde{M}_{\rm dR}(X) \times \widetilde{M}_{\rm dR}(X)
$$
such that the Riemann-Hilbert correspondence gives an equivalence of analytic groupoids
$$
\left(  \widetilde{M}_{\rm dR}(X), \, \srG_{\rm dR} \right) 
\;
\cong 
\; 
(\widetilde{M}_B(X)  ,\, \srG _B) .
$$
This extends in a natural way to an étale groupoid on the Hodge moduli space
$$
\srG_{\rm Hod} \rightarrow \widetilde{M}_{\rm Hod}(X) \times \widetilde{M}_{\rm Hod}(X) .
$$
\end{theorem}

The isomorphism of topological spaces between
$X^{\rm top}$ and $\overline{X}^{\rm top}$ gives an equivalence
$$
(\widetilde{M}_B(X)  ,\, \srG _B) \; \cong \; (\widetilde{M}_B(\overline{X})  ,\, \srG _B).
$$
Using the Riemann-Hilbert correspondence we can then make a Deligne glueing.  In terms
of groupoids this can be viewed as follows: 
let 
$$
\widetilde{M}_{\rm DH} := \widetilde{M}_{\rm Hod}(X) \sqcup \widetilde{M}_{\rm Hod}(\overline{X})
$$
with Hecke-gauge groupoid $\srG _{\rm DH}$ combining the $\srG_{\rm Hod}$ on both pieces, together with
pieces identifying points that correspond to elements of the Betti moduli space that are identified
under the previous equivalence:
$$
M_{\rm DH} = \left( \widetilde{M}_{\rm DH}, \, \srG _{\rm DH} \right)  .
$$
The quotient of this groupoid would be some kind of non-separated analytic stack but with stabilizer groups that aren't
very well behaved. Although we don't identify the precise framework for such a quotient, we note that the 
tangent bundle of the ``quotient'' $M_{\rm DH}(X)$ may be defined, since the groupoid is étale. This gives, in 
particular, the pullback of the tangent
bundle by a section. 

\begin{theorem}
\label{maintwistor}
Suppose $(E_X,\partial, \overline{\partial}, \varphi , \varphi ^{\dagger} ,h)$ is a tame harmonic bundle on $X$, 
satisfying our Hypothesis \ref{hypstar}, and choose a framing $\beta: E_x\cong \cc^2$
taking $h_x$ to the standard hermitian metric on $\cc^2$.  Then we get in a natural way a section of the Deligne-Hitchin groupoid
$$
\rho : \pp^1\rightarrow \left(  \widetilde{M}_{\rm DH}(X), \, \srG_{\rm DH} \right)  .
$$
Let $T_{\rho}$ be the pullback by $\rho$ of the tangent bundle of $M_{\rm DH}(X)$. It has a filtration 
$$
0\subset W_0T_{\rho} \subset W_1T_{\rho}\subset W_2T_{\rho} = T_{\rho}
$$
where $W_1T_{\rho}$ is the set of tangent vectors that preserve the eigenvalues of the residues, and $W_0T_{\rho}$ is the tangent space
of the change of framing. Then $(T_{\rho},W_{\cdot})$ is a {\em mixed twistor structure}, meaning that 
$W_k/W_{k-1}$ is a semistable bundle on $\pp^1$ of slope $k$ (for $k=0,1,2$). 
\end{theorem}

The weight $1$ property of the graded piece corresponding to deformations fixing the eigenvalues of the residues,
corresponds to the fact---well-known in the physics literature---that moduli spaces of flat bundles with fixed conjugacy classes 
have a Hitchin-type hyperkähler structure,
see for example \cite{Holdt}. 

We'll base our proof on the pure twistor ${\mathcal D}$-module theory of \cite{Sabbah,Mochizuki}. The more general and full theory
of mixed twistor ${\mathcal D}$-modules
\cite{MochizukiMTM} should allow for a more direct proof. Mochizuki has communicated the
suggestion to consider the mixed twistor ${\mathcal D}$-module
${\mathfrak T}[\ast D][!x]$ where ${\mathfrak T}$ is the extension associated to $End(E)$; one would still need to show a
compatibility with the moduli space construction.

\section{Logarithmic connections}

Throughout this paper, $Y$ is a smooth compact Riemann surface and $D = \{ y_1,\ldots , y_k\}$ is a nonempty reduced divisor. 
We set $X:= Y-D$ and fix a basepoint $x\in X$. 

A quasi-parabolic logarithic $\lambda$-connection of rank $2$ consists of $\lambda \in \cc$, a vector bundle $E$ of rank $2$ over $Y$
together with a logarithmic $\lambda$-connection operator 
$$
\nabla : E\rightarrow E\otimes _{\Oo _Y} \Omega ^1_Y(\log D)\;\;\;\; \nabla (ae) = a\nabla (e) + \lambda (da) e
$$
and for each $y\in D$ a subspace $F_y\subset E_y$ of rank $1$, preserved by the residue ${\rm res}_y(\nabla )$. 

A {\em framing} at the basepoint $x\in X$ is an isomorphism $\beta : E_x \cong \cc^2$. 

These all have versions relative to a base scheme $S$ where $\lambda :S\rightarrow \aaaa^1$ and  
$E$ becomes a bundle on $Y\times S$. 

We usually consider Hypothesis \ref{hypstar} on $(\lambda , E, \nabla , F, \beta )$, implying in particular that the framed object is rigid.

Let $\srMtilde _{\rm Hod} (Y,\log D,x)$ denote the moduli functor of quasi-parabolic logarithmic $\lambda$-connec\-tions
of rank $2$ on $(Y,D)$, framed at $x$ and satisfying Hypothesis \ref{hypstar}. 
This functor associates to a scheme $S$ the set of relative data $(\lambda , E,\nabla , F, \beta )$
on $Y\times S$, up to isomorphism. It maps via $\lambda$ to $\aaaa^1$. 

Construction of a moduli space is by now a classical subject. Some references, including a few further directions, are \cite{BalajiSeshadri,Bhosle,BhosleRamanathan,BiswasInabaKomyoSaito,HaiEtAl,Herrero,InabaIwasakiSaito,
Jeffrey,Konno,MaruyamaYokogawa,NitsureCoh,NitsureLog,Singh}  
but it would be impossible to mention all of the relevant articles here. 

\begin{proposition}
This moduli functor $\srMtilde _{\rm Hod} (Y,\log D,x)$ is represented by a separated algebraic space 
$\Mtilde _{\rm Hod} (Y,\log D,x)$
that is  locally of finite type over $\aaaa^1$.
\end{proposition}
\begin{proof}
We can cover the moduli functor by subsets $\srU _k$ consisting of points where the maximum degree of a subbundle of $E$ is $k$. 
These are bounded. The usual theory 
allows us to represent each of these as a quotient of a quasiprojective scheme by the 
action of  a group of the form $GL(N)$ for some large $N$. Since objects are rigid by Condition (1) of Hypothesis \ref{hypstar},
the stabilizer groups are trivial. Luna's etale slice theorem implies that the quotient is an
algebraic space. 
The union of these spaces of finite type is locally of finite type, although not of finite type \cite[Lemma 4.13]{Herrero}. 
 
To show separatedness, we use the condition (2) of Hypothesis \ref{hypstar}. Note that 
the moduli space is covered by finite type algebraic spaces and we may assume given two curves in one of those. That is to say,
we assume given a pointed smooth curve $(S,0)$ and denote by $S_{\eta}:= S-\{ 0\}$, 	and we are given two maps $S
\rightarrow \srU$ that agree on $S_{\eta}$. We obtain two bundles $E$ and $E'$
on $Y\times S$ (together with all their data) that are isomorphic over $Y\times S_{\eta}$, and suppose either $\lambda (0)\neq 0$ or
$\nabla _0$ and $\nabla '_0$ have irreducible spectral curves. We would like to show that the two maps agree on $S$,
in other words that the isomorphism extends to $Y\times S$. 

In the case $\lambda (0)\neq 0$, reason complex-analytically. 
The framed representation space of the fundamental group of $X=Y-D$ is separated, so 
the isomorphism extends to an isomorphism of flat holomorphic bundles on $X\times S$, hence to an isomorphism of
bundles on $Y\times S$ by Hartogs' theorem. 
Again by Hartogs' theorem the connection operator also extends, and the framing extends since $x\in X$. 
Our isomorphism preserves the subbundles $F_{\{ y\} \times S}$ away from the origin of $S$, it follows that these subbundles are also preserved at
$0\in S$. Although the construction was analytic, the resulting isomorphism of bundles is algebraic since $Y$ is proper. 

We need to treat the case $\lambda (0)=0$ assuming the spectral varieties of $\nabla (0)$ and $\nabla '(0)$ are irreducible. 
Let $t$ denote the coordinate on $S$. Let $g:E_{\eta} \cong E'_{\eta}$ be the isomorphism of bundles respecting
connections. There is a power of $t$ so that $t^ag$ extends to a morphism 
$$
t^ag : E \rightarrow E'
$$
that is nonzero for $t=0$. But this morphism respects the connections, i.e.\  the Higgs fields at $t=0$. Since the spectral
varieties are irreducible, this implies that the morphism is an isomorphism over $t=0$ too. Now the condition that $g$ respects the
framing impliess that $a=0$ so $g$ extends to an isomorphism as required. As before, it respects the sub-bundles $F_{\{ y\} \times S}$.
\end{proof}

This proposition is the construction of moduli spaces for Theorem \ref{mainmod}, see Corollary \ref{smoothcor} for smoothness.

We note that the rescaling of a $\lambda$-connection to $\lambda ^{-1}\nabla$ provides an isomorphism on the open set $\Gm \subset \aaaa^1$
where $\lambda \neq 0$
$$
\Mtilde_{\lambda \neq 0}(Y,\log D, x) \stackrel{\cong}{\longrightarrow} \Gm \times \Mtilde_{\lambda =1}(Y,\log D, x)  .
$$
This is equivariant for the $\Gm$-action by rescaling on the left and the trivial action on the right hand side. Here and below,
we write $\Mtilde_{\lambda \neq 0}$ etc. for the fibers of $\Mtilde_{\rm Hod}$ over various values of $\lambda$.

\subsection{Riemann-Hilbert morphism}

For each $y\in D$ choose a point $\eta _y\in X$ near $y$. Given a local system $L$ on $X$, the {\em nearby fiber} to $y\in D$ is the
fiber $L_{\eta _y}$, and it has a local monodromy operator induced by the loop based at $\eta _y$ going once around $y$. 

Let $\Mtilde_B(X,D,x)$ denote the moduli space of local systems $L$ on $X$ provided with a subspace $F_y$ of the nearby fiber at
each $y\in D$, invariant under the local monodromy, such that Hypothesis \ref{hypstar}
is satisfied, together with a framing $\beta : L_x\cong \cc ^2$.

Consider the open subset of points called {\em nonresonant}
$$
\Mtilde_{\lambda \neq 0}^{\rm nr}(Y,\log D,x) \subset \Mtilde_{\lambda \neq 0}(Y,\log D,x)
$$
defined to be the set of $(\lambda , E,\nabla , F, \beta )$ such that, if $a_y$ denotes the eigenvalue of ${\rm res}_y(\nabla )$ on $F_y$
and $b_y$ the eigenvalue on $E_y/F_y$, then 
$$
b_y - a_y \not \in \lambda \cdot  \zz _{<0} 
$$
for any $y\in D$. 

Let $U_y$ be a small disk around $y$. If $(\lambda , E,\nabla , F, \beta )$  satisfies the non-resonance condition, then letting
$L:= E^{\lambda ^{-1} \nabla }$ denote the local system of flat sections, there is a unique 
sub-connection of rank $1$
$$
F_{U_y} \subset E|_{U_y}
$$
whose fiber over $y$ is $F_y$. We denote by $F_{L,y}$ the fiber of $F_{U_y}$ at the nearby point $\eta _y$. 
Hypothesis \ref{hypstar} implies that $(L,F)$ has only scalar endomorphisms (the condition for inclusion in $\Mtilde _B$). 
This serves to define
the {\em Riemann-Hilbert morphism}.  Beyond Deligne \cite{Deligne} some references
include \cite{Boalch,BudurLererWang,InabaIwasakiSaito,NitsureRS,NitsureSabbah}, 
although it would again be impossible to give a complete list.

\begin{proposition}
The Riemann-Hilbert correspondence is a morphism of complex analytic spaces
$$
\Mtilde _{\lambda \neq 0}^{\rm nr}(Y,\log D,x) \stackrel{RH}{\longrightarrow} \Gm \times \Mtilde_B(X,D,x)  
$$
sending $(\lambda , E, \nabla , F, \beta )$ to $(\lambda, L,F_L,\beta )$ where $L:= E^{\lambda ^{-1} \nabla }$ and $\beta$ is the same framing on 
$L$ as on $E$, and $F_{L,y}$ is defined as in the previous paragraph. 
\end{proposition}
\eop

\subsection{Morphisms and equivalences of groupoids}

If $Z$ is an algebraic or analytic space, a {\em groupoid} acting on $Z$ is an algebraic (resp. analytic) space $\srG$ with maps
$$
(s,t):\srG \rightarrow Z \times Z, \;\;\; m:\srG \times _Z \srG \rightarrow \srG, \;\;\; e:Z \rightarrow \srG , \;\;\; i:\srG \rightarrow \srG
$$
making $(Z,\srG )$ into a groupoid in the category of algebraic (resp. analytic) spaces. If $Z'$ is another space, then a map 
$Z'\rightarrow (Z,\srG )$ consists of an open covering $\{ U_i\}$ of $Z'$, morphisms $z_i:U_i \rightarrow Z$, and morphisms 
$g_{ij}:U_{ij} \rightarrow \srG$ compatible with the previous by the source and target maps, and satisfying a cocycle condition.
We may similarly define the notion of a map of groupoids 
$(Z',\srG ') \rightarrow (Z,\srG )$: a quick way is to view it as a functor, defined after possibly replacing 
$Z'$ by an open covering. The previous definition may be seen in this manner. 

A {\em natural isomorphism} between two maps is a natural transformation of functors after refinement, 
given in the concrete notation by a common refinement $\{ U''_k\}$ of the two coverings plus 
a collection of maps $U''_k\rightarrow \srG$ satisfying 
the natural intertwining condition between the source and target map data. 
The set of maps as objects related by natural isomorphisms forms a groupoid 
$Hom((Z',\srG '),(Z,\srG))$ in the algebraic sense. In general, elements here can have automorphisms, but in 
the case where $(s,t)$ is a monomorphism (as shall be the case for us over $\lambda \neq 0$) the groupoid of maps
is equivalent to a set. 

A map is an {\em equivalence} if there is a quasi-inverse, that is to say a map going in the
opposite direction such that the two compositions are equivalent by natural isomorphisms to the identities. These 
standard definitions (going back to Ehresmann and Satake) 
serve as a replacement for taking some kind of quotient stack of the groupoid, allowing us to sidestep the issue of how precisely to
view the quotient stacks. 

One says that the groupoid is {\em étale} if the source and target maps $\srG \rightarrow Z$ are etale, and in this case that the groupoid
is {\em smooth} if $Z$ is smooth. We can  then define the tangent bundle, and maps have differentials in the usual way. 

\subsection{Hecke-gauge groupoid on residues}

For $y\in D$ we define the {\em residual space at $y$} to be $R_y:= \aaaa^1 \times \cc^2$ with coordinates noted $(\lambda , a_y ,b_y)$.
Of course this is just $\cc^3$, the notation is meant to distinguish the $\lambda$-direction from the two residual ones. 
If $(\lambda , E,\nabla , F, \beta )\in \Mtilde _{\rm Hod} (Y,\log D, x)$ then we get the residue
$$
{\rm Res}_y(\lambda , E,\nabla , F, \beta ) = (\lambda , a_y ,b_y) \in R_y
$$
where $a_y$ is the eigenvalue of ${\rm res}_y(\nabla )$ on $F_y$ and $b_y$ is the eigenvalue on $E_y/ F_y$. Set 
$$
R := R_{y_1} \times _{\aaaa^1} \cdots \times _{\aaaa^1} R_{y_k}
$$
and we obtain the residue vector ${\rm Res}(\lambda , E,\nabla , F, \beta ) =(\lambda , a, b)\in R$. 

We define a residual Hecke-gauge groupoid acting on $R_y$, in a way intended to be compatible with the Hecke-gauge groupoid that we'll define on 
$\Mtilde_{\rm Hod}$ below. Consider the operations (the third one being only partially defined)
$$
{\bf h}_y, {\bf h}^{-1}_y, {\bf p}_y : R_y\rightarrow R_y
$$
given by
$$
{\bf h}_y(\lambda , a ,b) = (\lambda , b-\lambda ,a),\;\;\;
{\bf h}^{-1}_y(\lambda , a ,b) = (\lambda , b ,a + \lambda ),
$$
and
$$
{\bf p}_y(\lambda , a ,b) = (\lambda , b,a) \;\; \mbox{ defined when } \;\; a\neq b.
$$
This latter means that the graph of ${\bf p}_y$ is the open subset of $R$ complement of the diagonal. 

Let $\srG _{R,y}$ be the groupoid generated by these operations subject to the relations that ${\bf h}_y^{-1}$ is inverse to ${\bf h}_y$, 
${\bf p}_y^2 = 1$, and ${\bf p}_y$ commutes with ${\bf h}_y^2$.

Here is a more explicit description. Note that ${\bf h}_y {\bf p}_y(\lambda , a,b)= (\lambda , a-\lambda,b)$ defined when $a\neq b$.  
We have $({\bf h}_y {\bf p}_y)^{k} (\lambda , a,b)= (\lambda , a-k \lambda , b)$ defined on the open subset where
$$
b\not \in \{ a,a-\lambda , \ldots , a-(k -1)\lambda  \} .
$$
For 
$\epsilon = 0,1$, $k \in \nn$ and $m \in \zz$ let 
$$
{\bf g}_y(\epsilon , k, m ) := {\bf p}_y^{\epsilon} ({\bf h}_y {\bf p}_y)^{k} {\bf h}_y^{m} ,
$$
and we can  similarly write down the open subsets of definition of these operations depending on $\epsilon , k, m$. 
The graph is a locally closed subset ${\rm Graph}({\bf g}_y(\epsilon , k, m)\subset R_y\times R_y$, isomorphic to
the open subset of definition of ${\bf g}_y(\epsilon , k, m ) $ in $R_y$. 

\begin{lemma}
This groupoid $\srG _{R,y}$ is étale, and has the following description as a disjoint sum of locally closed subvarieties
of $R_y\times R_y$:
$$
\srG_{R,y} = \coprod _{\epsilon , k , m} {\rm Graph}\left( {\bf g}_y(\epsilon , k, m) \right) .
$$
The map $\srG _{R,y} \rightarrow R_y \times_{\aaaa^1} R_y$ is a monomorphism over $\lambda \neq 0$. 
\end{lemma}
\eop

Now define $\srG _R$ to be the product groupoid acting on $R$, where the operations at different points commute. 
It is an étale groupoid with an analogous structure statement, and again $\srG _R \rightarrow R\times _{\aaaa^1} R$ is
a monomorphism. Notice that, although the image is a closed subset,  
$\srG_R$ is {\em not} isomorphic to its image. 

The quotient $R/\srG _R$ would be a non-separated space somewhat along the lines of
the ``bug-eye'' spaces introduced by Kollár in \cite{Kollar}.  

We may similarly define the Betti residual space $R_{B,y} = \cc ^{\ast} \times \cc^{\ast}$, then $R_B:= \prod _{y\in D} R_{B,y}$, and
we let $\srG _{R_B}$ be the groupoid generated by the partially defined operations ${\bf p}_y(\alpha _y, \beta _y):= (\beta_y,\alpha _y)$
defined when $\alpha _y \neq \beta _y$, subject to the relations ${\bf p}_y^2= 1$ and for different values of $y$ they commute. 

\begin{lemma}
We have an analytic equivalence of groupoids 
$$
(R,\srG _R)^{\rm an} _{\lambda \neq 0}\stackrel{\cong}{\longrightarrow} \left( \Gm \times (R_B, \srG _{R_B} )  \right) ^{\rm an}
$$
given by
$$
(\lambda , a,b) \mapsto (\lambda , e^{2\pi i a}, e^{2\pi i b} ).
$$
\end{lemma}
\eop

\subsection{Hecke-gauge groupoid}

Define, at each point $y\in D$, an operation $H_y$, its inverse $H^{-1}_y$ and a partially defined operation 
$P_y$ 
on  $\Mtilde_{\rm Hod} (Y,\log D,x)$. The first $H_y$ is the well-known {\em Hecke
operation}, or {\em elementary transformation}, as has been considered in 
\cite{InabaIwasakiSaito} and more recently \cite{FassarellaLoray,HuHuangZong,Matsumoto}. 
Given $(\lambda , E, \nabla , F, \beta )$, we set 
$$
E':= {\rm ker} \left (E \rightarrow E_y / F_y \right)
$$
and let $F'_y$ be the image of $E(-y)_y$ in $E'_y$. At the other points $y_i\neq y$ keep the same $F'_{y_i}:= F_{y_i}$;
Note that $E'|_{X}=E|_{X}$.
The condition that $F$ is preserved by the residue of $\nabla$ implies that 
$\nabla |_{X}$ extends to a $\lambda$-connection $\nabla '$ on $E'$, again with residue preserving $F'$. We can take $\beta ':= \beta$
since $x\in X$. 

Let $T_y(E):= E\otimes \Oo _Y(y)$ with the induced subspaces, and put $H_y^{-1}:= T_yH_y=H_yT_y$. This is inverse to $H_y$. 

The operation $P_y$ is going to be only partially defined. Let 
$$
\Mtilde _{\rm Hod}(X,\log D,x)^{(y)}\subset \Mtilde_{\rm Hod} (X,\log D,x)
$$
denote the open subset consisting of points where the eigenvalues of ${\rm res}_y(\nabla )$ are distinct. 
We define
$$
P_y: \Mtilde_{\rm Hod} (Y,\log D,x)^{(y)}\rightarrow  \Mtilde_{\rm Hod} (Y,\log D,x)
$$
to be the operation that replaces the subspace $F_y$ by its complementary eigenspace of ${\rm res}_y(\nabla )$,
keeping the same $F$ at the other points of $D$.

Define the {\em Hecke-gauge groupoid}  $\srG_{\rm Hod}$ to be the groupoid of operations on $\Mtilde_{\rm Hod} (Y,\log D,x)$ 
generated by the operations $H_y$, $H^{-1}_y$ and 
$P_y$ subject to the relations that the first two are inverses, that $P_y^2=1$, that $P_y$ commutes with
$H_y^2$, and that the operations for
different values of $y$ commute.

\begin{proposition}
This defines an étale groupoid with structural map
$$
\srG_{\rm Hod} \rightarrow \Mtilde_{\rm Hod} (Y,\log D,x)\times _{\aaaa^1} \Mtilde_{\rm Hod} (Y,\log D,x)
$$
that is a monomorphism over $\lambda \neq 0$. 
The residue gives a morphism of groupoids
$$
(\Mtilde_{\rm Hod} (Y,\log D,x) , \srG_{\rm Hod} ) \stackrel{Res}{\longrightarrow} (R,\srG _R).
$$
\end{proposition}
\begin{proof}
Hypothesis \ref{hypstar} is preserved by our operations. 
As before, one may consider the composed operations 
$$
G_y(\epsilon , k, m) := P_y^{\epsilon} (H_yP_y)^k H_y^m
$$
for $\epsilon = 0,1$, $k\in \nn$ and $m\in \zz$. The domain of definition of $G_y(\epsilon , k,m)$
is the pullback, under the residue at $y$, of the domain of definition of ${\bf g}_y(\epsilon , k,m)$. 
Using the given relations, any element of $\srG_{\rm Hod}$ may be expressed uniquely as a product over $y\in D$ 
of $G_y(\epsilon , k,m)$,
and this gives the expression
$$
\srG _{\rm Hod} = \prod _{y\in D} \left[  
\coprod _{\epsilon ,k,m} {\rm Graph}\left( G_y(\epsilon , k, m)\right) 
\right] .
$$
We see that $\srG_{\rm Hod}$ is an étale groupoid. 
The operations $H_y,H^{-1}_y,P_y$ on $\Mtilde _{\rm Hod}(X,\log D,x)$ are compatible with the
operations ${\bf h}_y, {\bf h}^{-1}_y, {\bf p} _y$ on residues, so the residue map descends to a map 
of groupoids. 

The $({\rm source}, {\rm target})$ structural map is a monomorphism 
over $\lambda \neq 0$, since the different pieces of the decomposition
map to the corresponding pieces of $\srG_R$ and these are disjoint. 
\end{proof}

The Betti gauge groupoid is defined in a corresponding way. 
Let $\Mtilde _B(X,D,x)^{(y)}$ denote the subspace where the eigenvalues
of local monodromy operator at $y\in D$ are distinct, and let 
$P_{B,y}: \Mtilde _B(X,D,x)^{(y)} \rightarrow \Mtilde _B(X,D,x)^{\ast}$ denote the operation of replacing the eigenspace $F_y$ by its complementary
eigenspace. Let $\srG _B$ be the groupoid acting on 
$\Mtilde _B(X,D,x)$, generated by the partially defined operations $P_{B,y}$ with relations $P_{B,y}^2=1$ and that  
for distinct values of $y$ the operations commute. 

This is again an étale groupoid, and the local monodromy maps give a morphism of groupoids 
$$
(\Mtilde _B(X,D,x) , \srG _B) \rightarrow (R_B, \srG _{R_B}).
$$
In this case, the groupoid has an algebraic space quotient
$$
\Mtilde _B(X,D,x)\rightarrow M_B(X,D,x) / \srG _B 
$$

\begin{theorem}
\label{thmRH}
The Riemann-Hilbert correspondence gives an equivalence of analytic groupoids
$$
\left( \Mtilde _{\lambda \neq 0}(Y,\log D,x) , \srG \right) 
\stackrel{\cong}{\longrightarrow}  \Gm \times 
\left( \Mtilde _B(X,D,x) , \srG _B \right)
$$
compatible with the residue morphisms. 
\end{theorem}
\begin{proof}
It suffices to consider the fiber over $\lambda = 1$. 
Using the operations of $\srG_{\rm Hod}$ and tensoring with a rank $1$ connection, 
we may move any small neighborhood in 
$\Mtilde _{\lambda \neq 0}(Y,\log D,x) $ into a small neighborhood where the residues 
$(a_y,b_y)$ lie in $U_y\times V_y$ for $U_y,V_y\subset \cc$ small neighborhoods on which the
logarithm is well defined. These map isomorphically to the corresponding neighborhoods in 
$\Mtilde _B(X,D,x)$ by \cite{Deligne}. The action of $\srG _B$ corresponds to that of $\srG_{\rm Hod}$
whenever $U_y$ and $V_y$ overlap. 
\end{proof}

This completes the proof of Theorem \ref{mainRH}.

\subsection{The Deligne-Hitchin twistor space and preferred sections}

Let 
$$
\Mtilde _{\rm DH} := 
\Mtilde_{\rm Hod} (X,\log D , x) \sqcup 
\Mtilde _{\rm Hod}(\overline{X}, \log \overline{D},\overline{x}).
$$
On this, we have a groupoid $\srG _{\rm DH}$ defined as the disjoint union of 
$\srG_{\rm Hod}$ for $X$ with the same for $\overline{X}$, together with 
the pieces defining the identification obtained by Theorem \ref{thmRH} from
$$
\left(
\Mtilde_B^{\rm an}(X,D,x) ,  \srG _{B,X} \right) 
\cong 
\left( 
\Mtilde _B^{\rm an}(\overline{X},\overline{D},\overline{x}) ,  \srG _{B,\overline{X}}  \right) .
$$
We obtain a complex analytic space with etale groupoid
$$
\left( M_{\rm DH}  , \srG _{DH} \right)  \rightarrow \pp ^1
$$

\begin{theorem}
\label{prefsec}
Suppose $(E_X,\partial, \overline{\partial}, \varphi , \varphi ^{\dagger} ,h)$ is a tame harmonic bundle of rank $2$ on $X$ 
satisfying the conditions of Hypothesis (\ref{hypstar}. Fix a framing $\beta : E_x\cong \cc^2$. 
Then we obtain a 
canonical corresponding {\em preferred section} to the groupoid
$$
\rho :\pp^1 \rightarrow \left( \Mtilde _{\rm DH}  , \srG _{DH} \right) 
$$
yielding the standard construction of family of Higgs bundles and $\lambda$-connections over $X$. 
\end{theorem}
\begin{proof}
Suppose $y\in D$. Let $u=(a,\alpha )$ and $u'=(a',\alpha ')$ be the two parabolic weight and residual eigenvalue pairs 
at $\lambda = 0$ for the point $y$. Hypothesis \ref{hypstar} (3) says that these KMS spectrum elements 
are distinct modulo $\zz$ i.e.\ $u-u'\not \in \zz \times \{ 0\}$. 
Recall the formulas from \cite{Mochizuki}, given in the introduction as \eqref{frakp} \eqref{frake}
giving the parabolic weight ${\mathfrak p}(\lambda )$ and eigenvalue ${\mathfrak e}(\lambda )$ corresponding
to $u$, and ${\mathfrak p}'(\lambda )$ and ${\mathfrak e}'(\lambda )$ corresponding to $u'$.

The main problem, related to what Mochizuki \cite{Mochizuki} calls ``difficulty (b)'' and to Sabbah's picture \cite[page 70]{Sabbah}, 
is that the parabolic weights might become equal, for some values of $\lambda$, and similarly, the eigenvalues could become
resonant. Luckily, these two things don't happen simultaneously. Indeed 
the modifications given by \eqref{frakp} and \eqref{frake} are always bijective. Thus, 
the pairs 
$({\mathfrak p}(\lambda ),  {\mathfrak e} (\lambda ))$ and $({\mathfrak p}'(\lambda ),  {\mathfrak e} '(\lambda ))$ 
are distinct modulo the action of $\zz$ for each $\lambda$.

This allows us to use the groupoid $\srG_{\rm Hod}$ to move around enough
to define the quasiparabolic structures in a holomorphically varying way.

Look at a small neighborhood $\lambda _0\in U \subset \pp^1$. We may assume one of two cases: either
\newline
(a)\, 
${\mathfrak p}(\lambda ) \neq {\mathfrak p}'(\lambda )$ in $\rr / \zz$ for all $\lambda \in U$, or 
\newline
(b)\, ${\mathfrak p}(\lambda _0) = {\mathfrak p}'(\lambda _0) + k$, $k\in \zz$, but 
${\mathfrak e} (\lambda )\neq {\mathfrak e} '(\lambda ) -k\lambda$ for all $\lambda \in U$. 

We can assume (possibly reducing the size of $U$) that there is some $b\in \rr / \zz$ that is distinct from 
${\mathfrak p}(\lambda )$ and ${\mathfrak p}'(\lambda )$ for all $\lambda \in U$. This leads to a family of 
bundles with logarithmic $\lambda$-connections $(E(\lambda ), \nabla (\lambda ))$ for $\lambda \in U$. 

For $y\in D$, in case (a), the parabolic structure on $E(\lambda )$ has distinct parabolic weights for all $\lambda$, so we get a
rank $1$ subbundle of $E_y$ (these are a bundle with subbundle in terms of the parameter $\lambda$). 
In case (b) at $y$, the two eigenvalues of the residue on $E(\lambda )$  are
${\mathfrak e} (\lambda )$ and ${\mathfrak e} '(\lambda ) -k\lambda$; since they are distinct
we can choose in a uniform way one of the
two eigenspaces of the residue of $\nabla (\lambda )$ over $\lambda \in U$. 

Either way, we obtain a rank $1$ subbundle of $E_y$ required to define a quasi-parabolic structure.

We note that at each point $\lambda$, 
the associated quasi-parabolic logarithmic $\lambda$-connection satisfies the conditions of Hypothesis \ref{hypstar}.
Indeed for $\lambda = 0$ the irreducibility of the spectral curve is an assumption on our harmonic bundle. In turn, this
implies that the harmonic bundle is indecomposable, so the associated canonical parabolic objects are stable for any $\lambda$. 
In our construction, the quasi-parabolic structure may be different so an argument is needed when $\lambda \neq 0$. 
There exists a set of parabolic weights for the quasi-parabolic structure that makes it stable. In case (a) we keep
the given parabolic weights, whereas in case (b) then we choose parabolic weights for
the subbundle that are very close to ${\mathfrak p}(\lambda ) \approx {\mathfrak p}'(\lambda )+k$. Stability of the canonical 
parabolic object implies stability of this parabolic object. 
Now, the only quasi-parabolic endomorphisms would be endomorphisms of the stable parabolic object so they are scalars, giving
part (1) of Hypothesis \ref{hypstar} in the case $\lambda \neq 0$.

Now that we know Hypothesis \ref{hypstar} holds, this gives the data required to define a
section $U\rightarrow \Mtilde _{\rm Hod}(X,\log D, x)$
if $U\subset \pp^1 - \{ \infty\}$, or similarly
$U\rightarrow \Mtilde _{\rm Hod}(\overline{X},\log \overline{D}, \overline{x})$ if $U\subset \pp^1 - \{ 0\}$. 

A different choice of $b$ and/or a different choice of one of the eigenvalues, leads to a section that differs by a 
section $U\rightarrow \srG _{\rm Hod} $. Therefore, on intersections of open sets $U$ these glue together
to give a well-defined section to the target modulo the groupoid, as stated. 
\end{proof}

This theorem gives the first part of Theorem \ref{maintwistor}.

\section{Tangent spaces and cohomology}

\subsection{Deformations of quasi-parabolic logarithmic connections}

The deformation theory is well-known, see for example \cite{Konno} and subsequent literature  for the Higgs case, or 
\cite{InabaIwasakiSaito} for parabolic logarithmic connections. 

Suppose $(E,\nabla , F,\beta )$ is a framed quasi-parabolic bundle with logarithmic $\lambda$-connection on,$(Y,D)$. We would like
to write down the complex governing its deformations. Let $End(E) = E^{\ast} \otimes E$ denote the endomorphism 
bundle. For $y\in Y$ we have a map $End(E) \rightarrow Hom(F_y,E_y/F_y)$. Combining these together, 
define the sheaf of {\em quasiparabolic endomorphisms} to be the kernel in the exact sequence
$$
0\rightarrow End_Q(E) \rightarrow End(E) \rightarrow \bigoplus _{y\in D} Hom(F_y,E_y/F_y) \rightarrow 0
$$
where the sheaf on the right is a direct sum of skyscraper sheaves located at the $y\in D$. Let $\Rr _y := End(F_y) \oplus End(E_y / F_y )$
denote the space of residue eigenvalues at $y$. We have a map $End_Q(E) \rightarrow \Rr _y$ and we define
the sheaf of {\em strongly quasiparabolic endomorphisms} to be the kernel in the exact sequence
$$
0 \rightarrow End_{SQ}(E) \rightarrow End_Q(E) \rightarrow \bigoplus _{y\in D} \Rr _y \rightarrow 0.
$$
Furthermore, recall that $x\in X$ is our basepoint; let $End_{Q,x}(E)$ be the kernel of $End_Q(E) \rightarrow End(E_x)$. 

Consider the Zariski tangent space $T:= T_{(E,\nabla , F, \beta )}\Mtilde _{\lambda} (Y,D,x) $. Let $W_0T\subset T$ be the
tangent space to the changes of framing, and let $W_1T$ be subspace of deformations that preserve the eigenvalues of 
the residue of $\nabla$ on $F_y$ and $E_y/F_y$. Set $W_{-1}T:= 0$ and $W_2T:= T$.

\begin{lemma}
\label{deflem}
The operator $\nabla$ induces a logarithmic $\lambda$-connection on $End(E)$ and this restricts to
an operator
$End_Q(E) \stackrel{d_{\nabla}}{\longrightarrow} End_{SQ}(E) \otimes \Omega ^1_Y(\log D)$.
We get a complex, and also, complexes from any compositions in the sequence
$$
\begin{array}{cccc}
End_{Q,x}(E)  &\longrightarrow \;\;\;\; End_Q(E) & & \\
& \downarrow  & & \\
&  End_{SQ}(E) \otimes \Omega ^1_Y(\log D)
& \rightarrow & End_{Q}(E) \otimes \Omega ^1_Y(\log D).
\end{array}
$$
The Zariski tangent space
$T=T_{(E,\nabla , F, \beta )}\Mtilde _{\lambda} (Y,D,x)$ 
to the moduli space at $(E,\nabla , F, \beta )$ is the first hypercohomology
$$
T = {\mathbb H}^1\left( 
End_{Q,x}(E) \stackrel{d_{\nabla}}{\longrightarrow} End_{Q}(E) \otimes \Omega ^1_Y(\log D) 
\right) .
$$
The subquotients $W_kT / W_nT$ are those induced by the various complexes obtained from the above sequence, for example
$$
W_1T / W_0T =  {\mathbb H}^1\left( 
End_{Q}(E) \stackrel{d_{\nabla}}{\longrightarrow} End_{SQ}(E) \otimes \Omega ^1_Y(\log D) 
\right) .
$$
\end{lemma}
\eop

\begin{lemma}
\label{obslem}
Suppose that $E$ has no strictly quasi-parabolic endomorphisms, i.e.\ there are no $\nabla$-invariant sections of $End_{SQ}(E)$. 
Then 
$$
{\mathbb H}^2\left( 
End_{Q}(E) \stackrel{d_{\nabla}}{\longrightarrow} End_{Q}(E) \otimes \Omega ^1_Y(\log D)  \right) = 0
$$
and the moduli functor is smooth at  $(\lambda , E,\nabla , F, \beta )$. Suppose the only quasi-parabolic endomorphisms are scalars. 
Then the sequences 
$$
0 \rightarrow \cc \rightarrow End(E_x) \rightarrow W_0T \rightarrow 0
$$
and
$$
0\rightarrow W_2T / W_1T \rightarrow \bigoplus_{y\in D} \Rr _y \rightarrow \cc \rightarrow 0
$$
are exact.
\end{lemma}
\begin{proof}
We may apply Serre duality to the complex 
$$
\left[ 
End_{Q}(E) \stackrel{d_{\nabla}}{\longrightarrow} End_{Q}(E) \otimes \Omega ^1_Y(\log D)  \right] .
$$
Some care must be taken with the fact that the differential is a first-order operator; the Serre dual becomes
$$
\left[
End_{Q}(E) ^{\ast} (-D) \stackrel{d_{\nabla}}{\longrightarrow} 
End_{Q}(E)^{\ast} \otimes \Omega ^1_Y 
\right] 
$$
$$
\cong  
\left[
End_{SQ}(E)  \stackrel{d_{\nabla}}{\longrightarrow} 
End_{SQ}(E) \otimes \Omega ^1_Y (\log D)
\right]  . 
$$
Thus, the ${\mathbb H}^2$ in the first part of the lemma is dual to the space of strictly quasi-parabolic endomorphisms,
yielding the first statement. Furthermore, ${\mathbb H}^2$ is the obstruction space for the deformation theory, so if it vanishes
the moduli functor is smooth. For the last part we use some exact sequences together with the hypothesis that 
$$
 {\mathbb H}^0\left( 
End_{Q}(E) \stackrel{d_{\nabla}}{\longrightarrow} End_{SQ}(E) \otimes \Omega ^1_Y(\log D) 
\right)  = \cc . 
$$
\end{proof}

\begin{corollary}
\label{smoothcor}
The moduli space $\Mtilde _{\rm Hod} (Y,\log D, x)$ and hence also the disjoint union 
$\Mtilde _{\rm DH} (Y,\log D, x)$
are smooth. 
\end{corollary}
\begin{proof}
Recall that $\Mtilde _{\rm Hod} $ was defined as the moduli space of objects satisfying Hypothesis \ref{hypstar}. 
Condition (1) of the hypothesis implies, since by assumption $D\neq \emptyset$, that there are no strictly quasi-parabolic endomorphisms,
so Lemma \ref{obslem} applies. 
\end{proof}

This corollary completes the proof of Theorem \ref{mainmod}.

Suppose given a tame harmonic bundle $(E_X,\partial, \overline{\partial}, \varphi , \varphi ^{\dagger} ,h)$ on $X$,
satisfying Hypothesis \ref{hypstar},
leading to a preferred section $\rho : \pp ^1 \rightarrow \left( \Mtilde_{\rm DH}  , \srG _{DH} \right) $
by Theorem \ref{prefsec}. By Corollary \ref{smoothcor}, $\Mtilde_{\rm DH}$ is smooth. 
We obtain the pullback by $\rho$ of the relative tangent bundle $T M_{\widetilde{DH}} / \pp^1$, that
is furnished with glueing data to obtain a bundle $T_{\rho}$ over $\pp^1$. 

Let $W_0T_{\rho}$ denote the subbundle of relative tangent vectors corresponding to change of framing $\beta$. 
Let $W_1T_{\rho}$ be the subbundle of relative tangent vectors whose projection to the tangent of the space
of residues $(R,\srG _R)$ is trivial, and let $W_2T_{\rho}:= T_{\rho}$ and $W_{-1}T_{\rho}:= 0$.
These definitions are compatible with the ones in Lemma \ref{deflem}. 

The goal of this section is to prove that
with this weight filtration, $T_{\rho}$ becomes a mixed twistor structure, in other words 
$W_kT_{\rho} / W_{k-1}T_{\rho}$ is a semistable bundle of slope $k$ on $\pp^1$. 
The idea is to apply \cite{Mochizuki, Sabbah}.

\subsection{Identification with a pure twistor cohomology}

We look in the neighborhood of a point $y\in D$. 

Let $u=(a,\alpha )$ and $u'=(a',\alpha ' )$ be the two KMS spectrum elements, in coordinates at $\lambda = 0$. 
Hypothesis \ref{hypstar} (3) says that $u-u' \not \in \zz \times \{ 0\}$. 
Let ${\mathfrak p}(\lambda )$ and ${\mathfrak e}(\lambda )$ be the parabolic
weight and eigenvalue at $\lambda$ corresponding to $u$, and 
${\mathfrak p}'(\lambda )$ and ${\mathfrak e}'(\lambda )$ corresponding to $u'$. 

The endomorphism bundle decomposes as a direct sum
$$
End(E) = \Oo \oplus S 
$$
were $S:=End^0(E)$  is the trace-free endomorphism bundle. 
This latter has rank $3$,
and the KMS spectrum elements are $(u-u'),0,(u'-u)$. 

The case of rank $1$ systems was dealt with in \cite{weighttwo} so we may focus here on the deformations
parametrized by the trace-free part $S$. 

Choose a point $\lambda _0 \neq 0$, and set
$$
p:= {\mathfrak p}(\lambda _0), \;\;\;\;
p':= {\mathfrak p}'(\lambda _0),  \;\;\;\;
e:= {\mathfrak e}(\lambda _0), \;\;\;\;
e':= {\mathfrak e}'(\lambda _0). 
$$
There are two basic cases, depending on whether $p-p'$ is an integer. 

\noindent
{\bf Case 1}---suppose $p-p'\in \zz$. 
The parabolic weights of $S$ are integers. Let 
$$
{\mathfrak e}''(\lambda ):= {\mathfrak e}'(\lambda ) - \lambda (p-p')
$$ 
be the eigenvalue corresponding to 
the transition of ${\mathfrak e}'(\lambda )$ from parabolic weight $p'$ to parabolic weight $p'':=p$, 
and let $e'' := {\mathfrak e}''(\lambda _0)$. 

We obtain the locally free sheaf in the parabolic structure $E:=E_{p+\epsilon}$ for some small $\epsilon$,
defined in a small neighborhood $U^{(\lambda _0)}$ of $\lambda _0$ in the $\lambda$-line. Put
$$
G_E:= E_{p+\epsilon} / E_{p-\epsilon} .
$$
It is a rank two bundle over $\{ y\} \times U^{(\lambda _0)}$. 
We recall from \cite{Mochizuki} that it has a decomposition according to the eigenvalues of the
residue. Those eigenvalues are, as functions of $\lambda$, 
$$
{\mathfrak e}(\lambda ),  \;\;\; {\mathfrak e}''(\lambda ) .
$$
The hypothesis of distinct KMS spectrum elements implies that $e\neq e''$, 
so we may identify the sections ${\mathfrak e}(\lambda )$ and ${\mathfrak e}''(\lambda )$ by
the notations $e$ and $e''$. The eigenspaces of ${\rm res}_y(\nabla )$ acting on $G_E$ may therefore be denoted by
$G_{E,e}$ and $G_{E,e''}$ and
$$
G_E = G_{E,e}  \oplus G_{E,e''}.
$$
This decomposition is valid over the neighborhood $U^{(\lambda _0)}$, so 
${\rm res}_y(\nabla )$ acts by ${\mathfrak e}(\lambda )$ on $G_{E,e}$ and by 
${\mathfrak e}''(\lambda )$ on $G_{E,e''}$.

Choose one of the two subspaces, say for example $G_{E,e}$, to be $F_y$ in the quasi-parabolic structure
over our neighborhood  $U^{(\lambda _0)}$.

For the trace-free endomorphism bundle 
we obtain a locally free sheaf from the parabolic structure $S=S_{\epsilon}$ for some small $\epsilon$,
defined the neighborhood $U^{(\lambda _0)}$ possibly reducing its size.  Put
$$
G:= S_{\epsilon} / S_{-\epsilon} .
$$
It is a rank three bundle over $\{ y\} \times U^{(\lambda _0)}$. It has a decomposition according to the eigenvalues of the
residue. We recall that there are three different eigenvalues here: 
$$
{\mathfrak e}(\lambda ) -{\mathfrak e}''(\lambda )  , \;\;\; 0 \;\;\; {\mathfrak e}''(\lambda ) -{\mathfrak e}(\lambda ) .
$$
Denote the corresponding subspaces by $G_{(e-e'')}$, $G_0$ and $G_{(e''-e)}$ so
$$
G = G_{(e-e'')} \oplus G_0 \oplus G_{(e''-e)}.
$$
We have locally free subsheaves
$$
Q':= End^0_{SQ}(E) \;\; \subset \;\; Q:= End^0_Q(E) \;\; \subset \;\; S= End^0(E)
$$
of trace-free endomorphisms that strictly preserve (resp. preserve) the quasi-parabolic structure. 
Along $\{ y\} \times U^{(\lambda _0)}$, the endomorphisms
not preserving the parabolic structure are the ones that send $G_{E,e}$ to $G_{E,e''}$, that is to say they have eigenvalue $(e''-e)$.
The ones in $Q$ that act trivially on the graded pieces are $Q'$, and these are the ones mapping to zero in $G_0$. 
Thus, we have  exact sequences
\begin{equation}
\label{esq}
0\rightarrow Q \rightarrow S_{\epsilon} \rightarrow G_{(e''-e)} \rightarrow 0
\end{equation}
and 
\begin{equation}
\label{esqprime}
0\rightarrow Q' \rightarrow S_{\epsilon} \rightarrow G_0\oplus G_{(e''-e)} \rightarrow 0.
\end{equation}

In order to define the bundle over $U^{(\lambda _0)}\subset \aaaa^1$ that corresponds to the Sabbah-Mochizuki
pure twistor structure on cohomology, following the discussion in the Appendix of \cite{Mochizuki},
we should consider the germs of holomorphic bundles of sections of the parabolic extension that are locally $L^2$
with respect to the Poincaré metric. Here these are germs around the point $\lambda _0$ in the $\lambda$-line. 

We'll denote these by $\Ll ^2$, and we have the complex with two bundles
$$
\Ll ^2 (S(\ast D)) \rightarrow \Ll ^2 (S\otimes \Omega ^1(\ast D)).
$$

In the present setting, $ \Ll ^2 (S)= \Ll ^2(S(\ast D)) $ is the sheaf of sections of $S_{\epsilon}$ 
whose projection into $G$ lies in the piece $G_0$,;
in other words we have an exact sequence
$$
0\rightarrow \Ll ^2 (S) \rightarrow S_{\epsilon} \rightarrow G_{(e-e'')} \oplus G_{(e''-e)} \rightarrow 0.
$$
We recall that this is due to the fact that a section projecting into one of the other pieces will go, on a half-disk centered at $\lambda _0$,
into a piece where the parabolic weight is $> 0$ so it would have a growth rate of $|z|^a$ with $a>0$ for values of
$\lambda$ in that half-disk of  $U^{(\lambda _0)}$. By definition here we are looking at germs around $\lambda _0$.

A similar description works for holomorphic one-forms, taking coefficients with logarithmic poles. 
However, in that case we should introduce the $W_{-2}$ term of the weight filtration on the $G_0$ piece. Here, all the 
KMS spectrum eigenspaces have rank $1$ so the weight filtration is trivial, thus 
$W_{-2}=0$ in the $G_0$ piece. This means
we have an exact sequence
$$
0\rightarrow \Ll ^2 (S\otimes \Omega ^1_Y(\ast D)) \rightarrow S_{\epsilon} \otimes \Omega ^1_Y(\log D) 
\rightarrow G_{(e-e'')} \oplus G_{(e''-e)} \oplus G_0 \rightarrow 0
$$
hence
$$
\Ll ^2 (S\otimes \Omega ^1_Y(\ast D))  = S_{-\epsilon} \otimes \Omega ^1_Y(\log D)  = S_{\epsilon} \otimes \Omega ^1_Y.
$$
A word about notation: these objects are all really over $U^{(\lambda _0)}$ but we don't write e.g. $D\times U^{(\lambda _0)}
\subset Y\times U^{(\lambda _0)}$ etc., and also the space of residues 
$\Omega ^1(\ast D)_y$ is a trivial bundle (over $U^{(\lambda _0)}$), a trivialization is chosen and used in the expressions
for example the one-prior exact sequence. 

Using the description of \cite{Mochizuki},
the bundle of cohomology of the pure twistor $\srD$-module has, as germ around $\lambda _0$, the 
first hypercohomology of the complex 
$\Ll ^2 (S) \rightarrow \Ll ^2 (S\otimes \Omega ^1_Y(\ast D))$ 
or isomorphically
$$
\left[ \Ll ^2 (S)  \rightarrow S_{-\epsilon} \otimes \Omega ^1_Y(\log D) \right] . 
$$

Let us compare this with the complex that governs the deformations of the quasi-parabolic logarithmic bundle 
$(E,\nabla , F)$. Recall from the previous discussion that this complex is 
$$
\left[ Q \rightarrow Q' \otimes \Omega ^1(\log D)  \right] 
$$
and that we have exact sequences \eqref{esq} and \eqref{esqprime}. 
Comparing with the exact sequences for the $L^2$ sheaves we get
$$
0\rightarrow  \Ll ^2(S) \rightarrow Q \rightarrow G_{(e-e'')} \rightarrow 0
$$
and 
$$
0\rightarrow  \Ll ^2(S\otimes \Omega ^1(\ast D)) \rightarrow Q' \otimes \Omega ^1(\log D) 
\rightarrow G_{(e-e'')} \rightarrow 0 .
$$
Therefore our two complexes fit into a diagram
$$
\begin{array}{ccc}
\Ll ^2 (S)  &\longrightarrow & \Ll ^2 (S\otimes \Omega ^1(\ast D))  \\
\downarrow & & \downarrow \\
Q&\longrightarrow & Q' \otimes \Omega ^1(\log D) \\
\downarrow & & \downarrow \\
G_{(e-e'')} & \stackrel{\cong}{\longrightarrow} & G_{(e-e'')} .
\end{array}
$$

\begin{corollary}
\label{case1cor}
In this Case 1, there is a natural quasi-isomorphism of complexes  
$$
\left[
\Ll ^2 (S)  \rightarrow  \Ll ^2 (S\otimes \Omega ^1(\ast D)) 
\right]
\;\;
\stackrel{q.i.}{\longrightarrow}
\;\;
\left[
Q  \rightarrow Q'\otimes \Omega ^1(\log D) 
\right] .
$$
\end{corollary}
\eop

\noindent
{\bf Case 2}---suppose $p-p'\not \in \zz$. . Then the parabolic weights of $S$ are not integers. 
This holds at $\lambda _0$ so we may choose our small neighborhood $U^{(\lambda _0)}$ so that it holds on
the neighborhood, and furthermore the neighborhood
is made small enough so that the various inequalities we'll state below also hold on $U^{(\lambda _0)}$.

Choose a representative (up to the Hecke-gauge group action) such that the parabolic weights $b,b-q$ of $E$, depending on $\lambda$, 
that are in the
interval $(-1,0)$ differ by $q<1/2$. 

We use $E:= E_{b(\lambda _0)+\epsilon}$ as the bundle, this is locally free. The parabolic weight subspace
$$
F_y:= E_{b+\epsilon -q} / E_{b+\epsilon - 1} \hookrightarrow E_{b+\epsilon } / E_{b+\epsilon - 1} = E|_{\{ y\} \times U^{(\lambda _0)}}
$$
will be used to define our quasi-parabolic structure over the neighborhood  $U^{(\lambda _0)}$.

Again  the bundle of trace-free endomorphisms $S= {\rm End}^0(E)$ has a three-step parabolic structure
having parabolic weights 
$$
-q \, < \, 0\, < \, q\; \in \; 
\left(
{\textstyle
-\frac{1}{2}, \frac{1}{2}
}
\right) .
$$
The sub-bundle of endomorphisms that preserve the quasi-parabolic structure is $Q=S_{\epsilon}$, whereas
the sub-bundle of those that also act trivially on the graded pieces is $Q'=S_{-q + \epsilon}$. This is uniform over
$\lambda \in U^{(\lambda _0)}$. 

The locally $L^2$ sections of $S$ (which are the same as those of $S(\ast D)$) are 
$$
\Ll ^2(S) = S_{\epsilon}
$$
since the KMS spectrum element has constant eigenvalue equal to $0$ here and the eigenspace has rank $1$ so it
is equal to its $W_0$ piece. The $W_{-2}$ piece is zero in the associated-graded, so 
$$
\Ll ^2(S\otimes \Omega ^1(\ast D))= W_{-2} S_{\epsilon} = S_{-\epsilon} = S_{-q+\epsilon }
$$
the last equality being because there are no parabolic weights between $-q$ and $0$. 

In this case, we conclude that the $L^2$ complex and the deformation theory complex are actually equal:

\begin{corollary}
\label{case2cor}
In the present Case 2, the complex of locally $L^2$ forms
is equal to the deformation complex for the quasi-parabolic structure
$$
\left[
\Ll ^2 (S)  \rightarrow  \Ll ^2 (S\otimes \Omega ^1(\ast D)) 
\right]
\;\; 
=
\;\;
\left[
Q  \rightarrow Q'\otimes \Omega ^1(\log D) 
\right] .
$$
\end{corollary}
\eop

\subsection{Globalization}

Now go back to the global case and put these two corollaries together. Let ${\mathbf H}^1({\rm End}^0({\mathfrak E}))$
denote the Sabbah-Mochizuki \cite{Sabbah,Mochizuki} pure twistor structure for the cohomology of the pure twistor $\srD$-module
corresponding to the trace-free endomorphisms of our harmonic bundle. For the central or scalar part, let $H^i_{DH}(X)$ denote the
twistor structure of weight $i$ associated to the pure Hodge structure on the cohomology of $X$. 

The relative tangent bundle of our Deligne-Hitchin groupoid is pulled back along the
preferred section to give $\rho ^{\ast}T(M_{\rm DH} / \pp ^1) $. We defined a three-step filtration
$$
0=W_{-1} \subset W_0\subset W_1 \subset W_2 = \rho ^{\ast}T(M_{\rm DH} / \pp ^1) 
$$
where $W_0$ is the deformations of the framing, $W_1/W_0$ is the deformations of the quasi-parabolic
logarithmic $\lambda$-connection conserving the eigenvalues of the residue, and $W_2/W_1$ is the
deformations of the residual data. The filtrations coming from the two $M_{\rm Hod}$ pieces glue
over  $\Gm$, because they coincide with the similarly defined filtrations on the tangent to $M_B$ 
pulled back along the preferred section. 

\begin{theorem}
We have a natural isomorphism 
$$
Gr_1^W\rho ^{\ast}T(M_{\rm DH} / \pp ^1) \cong {\mathbf H}^1({\rm End}^0({\mathfrak E})) \oplus H^1_{DH}(X).
$$
In particular, $Gr_1^W$ is a pure twistor structure of weight $1$. For $i=0,2$ we also have that 
$Gr_i^W$ is a pure twistor structure of weight $i$, in other words our tangent space with the given
weight filtration is a mixed twistor structure. 
\end{theorem}
\begin{proof}
In the previous subsections we have shown that there is a natural quasi-isomorphism between 
the complex calculating deformations of the quasi-parabolic logarithmic $\lambda$-connec\-tion,
and the complex of locally $L^2$ forms shown by \cite{Mochizuki} to calculate the twistor ${\mathbf H}^1$. This was
done for germs in the neighborhood of any $\lambda _0\in \aaaa^1$. 

It needs to be checked that this natural isomorphism is compatible with the glueing to the other
chart of the twistor $\pp^1$. In the Sabbah-Mochizuki theory, this glueing is done using the
sesquilinear pairing in Sabbah's definition of {\em $\srR$-triple} \cite{Sabbah}, 
whereas for the tangent of deformation theory it comes from comparison
with the Betti moduli spaces.

To make the comparison with \cite{Sabbah}, recall that an $\srR$-triple has two $\srR_X$-modules that are related
by a sesquilinear pairing. The first module is the minimal extension of the harmonic bundle (in this case, $S=End^0(E)$)
from $\aaaa^1 \times X$ to $\aaaa^1\times Y$, and the second one is the same for $\overline{X}$, but complex-conjugated
back to being an object over $\aaaa^1\times Y$. The sesquilinear pairing, defined over the unit circle $| \lambda | = 1$, 
takes values in distributions. The precise structure
is complicated near points of the divisor $D$ since the modules involve meromorphic sections, so this brings into play
the division of distributions \cite{SabbahDivision} and Mellin transform. Over points of $X$, the pairing is just the
same identification between local systems on $X$ and $\overline{X}$ that we are using. 

The higher direct image in \cite{Sabbah} is calculated using the Dolbeault resolution, then the pairing on the 
cohomology (i.e.\ higher direct image to a point)
involves wedging forms, pairing the $\srR_X$-module coefficients, and integrating \cite[\S 1.6.d]{Sabbah}. 

We can avoid having to look too closely at the behavior near singularities. This is because we need to understand the
Betti glueing vs the sesquilinear pairing, for classes in $H^1$.  It suffices to verify that the identifications are the same
for general values of $\lambda$ on the unit circle. We may therefore assume that the two residue eigenvalues at any
$y\in D$ don't differ by integer multiples of $\lambda$. 
Our cohomology space in question then has the property that
it is the image of the map 
$H^1_c(X, S^{\nabla _{\lambda}}) \rightarrow H^1(X,S^{\nabla _{\lambda}})$ 
from compactly supported cohomology to cohomology over $X$. We may therefore represent cohomology classes by
forms that are compactly supported on $X$, and pair them with other forms that are compactly supported on $X$
in order to check the identifications.

In this setting the pairing formula from \cite[\S 1.6.d]{Sabbah} is just the usual cup-product of cohomology classes via the
identification between de Rham cohomology of a $\lambda$-connection and Betti cohomology. That identification is
the one that occurs for the tangent spaces of our moduli spaces under the identification between tangent spaces, hence
deformation spaces, and cohomology spaces. This gives the compatibility. 

From the general theory of \cite{Sabbah,Mochizuki},
we get that $ {\mathbf H}^1({\rm End}^0({\mathfrak E})) $ is a pure twistor structure of weight $1$
meaning that as a bundle it is a direct sum of copies of $\Oo_{\pp^1}(1)$. We therefore obtain that required weight property
for $Gr_1^W\rho ^{\ast}T(M_{\rm DH} / \pp ^1) $. 

For the $Gr_0^W$ piece, it is easy to see that the deformations of change of framing give a trivial bundle over $\pp^1$
since the framing doesn't depend on anything. The space of deformations of change of framing,
globalized over the twistor line, is $End(\Oo _{\pp^1}^2) / H^0_{DH}(X)$
taking the quotient by the subspace $H^0_{DH}(X)=\Oo _{\pp^1}$ of scalar endomorphisms of the bundle. This has a weight $0$
twistor structure. 

For the $Gr_2^W$ piece, we refer to the discussion of \cite{weighttwo} for the weight two property of
the space of deformations of the residual data. Here again there is a modification by $H^2_{DH}(X)$,
namely we have an exact sequence
$$
0\rightarrow Gr_2^W \rightarrow \bigoplus _{y\in D} \Rr _{y,DH} \rightarrow H^2_{DH}(X)\rightarrow 0
$$
corresponding to the condition that the sums of all the residues should vanish. It is a condition happening on 
the determinant bundle. The morphism on the right is a morphism of weight $2$ twistor structures so
the kernel $Gr_2^W$ is a weight $2$ twistor structure. 

This concludes the proof that the full relative tangent space, along the preferred section, has a
mixed twistor structure with weights $0,1,2$. 
\end{proof}

It completes, in turn, the proof of Theorem \ref{maintwistor}. 

\bigskip

\noindent
{\bf Acknowledgements:}
This material is based upon work supported by a grant from the Institute for Advanced Study.
Supported by the Agence Nationale de la Recherche program 
3ia Côte d'Azur ANR-19-P3IA-0002, European Research Council Horizons 2020 grant 670624 (Mai Gehrke's DuaLL project) and the 
International Centre for Theoretical Sciences program ICTS/mbrs2020/02. I would like to thank the many colleagues who
have contributed,  through
numerous discussions and inspiring articles, to this work. 
I would like to thank Takuro Mochizuki, in addition to the previous, also for several improvements to a first draft of this work.

\

\noindent
Université Côte d'Azur, CNRS, LJAD, France

\noindent
\verb+carlos.simpson@univ-cotedazur.fr+

\end{document}